\documentclass{amsart}
\usepackage{amssymb}
\usepackage{amsthm}
\newcommand{\pn}{\par\noindent}
\newcommand{\pmn}{\par\medskip\noindent}

\newtheorem{theor}{Theorem}[section]
\newtheorem{prop}{Proposition}[section]
\theoremstyle{definition} 

 \theoremstyle{remark}
\newtheorem{rem}{Remark}[section]
\begin{document}
\title{Affine transformations of circle and sphere}
\author{Irina Busjatskaja \and Yury Kochetkov}
\begin{abstract} A non-degenerate two-dimensional linear
operator $\varphi$ transforms the unit circle into ellipse. Let
$p$ be the length of its bigger axis and $q$ --- the smaller. We
can define the deformation coefficient $k(\varphi)$ as $q/p$.
Analogously, if $\varphi$ is a non-degenerate three-dimensional
operator, then it transforms the unit sphere into ellipsoid. If
$p>q>r$ are lengths of its axes, then deformation coefficient
$k(\varphi)$ will be defined as $r/p$. In this work we compute the
mean value of deformation coefficient in two-dimensional case and
give an estimation of the mean value in three-dimensional case.
\end{abstract} \email{ibusjatskaja@hse.ru, yukochetkov@hse.ru}
\maketitle
\section{Introduction}
\pn This work is a continuation of the work \cite{BK}, where we
study the deformation of angles under the action of a linear
operator in $\mathbb{R}^2$. Here we study the deformation of the
unit circle and also made some comments about three-dimensional
case. \pmn Let $\varphi$ be a non-degenerate linear operator in
$\mathbb{R}^2$ and
$$A=\begin{pmatrix}a&b\\c&d\end{pmatrix},\quad ad-bc\neq 0$$ be
its matrix in standard base. Operator $\varphi$ transforms unit
circle $C$ into ellipse with canonical equation
$$\frac{x'^2}{p^2}+\frac{y'^2}{q^2}=1,\quad
p\geqslant q$$ in appropriate coordinate system $(x',y')$. The
number $q/p\leqslant 1$ will be called the deformation coefficient
$k(\varphi)$ of operator $\varphi$. \pmn In Section 2 we compute
the deformation coefficient $k(\varphi)$. \pmn \textbf{Theorem
2.1.}
$$k(\varphi)=\frac{\sqrt{a^2+b^2+c^2+d^2-\sqrt{(a^2+b^2+c^2+d^2)^2
-4(ad-bc)^2}}}{\sqrt{a^2+b^2+c^2+d^2+\sqrt{(a^2+b^2+c^2+d^2)^2
-4(ad-bc)^2}}}.$$ \pmn In Section 3 compute the mean value
$\overline{k}_2$ of coefficients $k(\varphi)$. \pmn
\textbf{Theorem 3.1.} $\overline{k}_2=3-4\ln(2)$. \pmn In Section
4 we demonstrate how to obtain a coarse upper bound for
$\overline{k}_2$: $\overline{k}_2<\frac 12$ (Theorem 4.1) with the
aim to generalize this result to three dimensional case:
$\overline{k}_3<\frac 13$ (Theorem 5.1).

\section{Deformation coefficient in two-dimensional case}
\pn Let $\varphi:\mathbb{R}^2\to\mathbb{R}^2$ be a non-degenerate
linear operator and
$$A=\begin{pmatrix} a&b\\c&d \end{pmatrix}$$ be its matrix
in the standard base. Operator $\varphi$ transforms unit
circle into ellipse with axes $p$ and $q$, $p>q$.
\begin{theor}
$$k(\varphi)=\frac qp=\sqrt{\frac{a^2+b^2+c^2+d^2-
\sqrt{(a^2+b^2+c^2+d^2)^2-
4(ad-bc)^2}}{a^2+b^2+c^2+d^2+\sqrt{(a^2+b^2+c^2+d^2)^2-
4(ad-bc)^2}}}.$$ \end{theor}
\begin{proof} If $\varphi^*$ is the conjugate operator, then $A^t$
is its matrix in the standard base. $p^2$ and $q^2$ are
eigenvalues of operator $\varphi^*\varphi$ with matrix
$$A^tA=\begin{pmatrix} a^2+c^2& ab+cd\\ ab+cd&
b^2+d^2\end{pmatrix}.$$ Thus, $p^2$ and $q^2$ are roots of $A^tA$
characteristic polynomial
$$s(x)=x^2-(a^2+b^2+c^2+d^2)x+(ad-bc)^2:$$
$$p=\sqrt{\frac{a^2+b^2+c^2+d^2+\sqrt{(a^2+b^2+c^2+d^2)^2-
4(ad-bc)^2}}{2}};$$
$$q=\sqrt{\frac{a^2+b^2+c^2+d^2-\sqrt{(a^2+b^2+c^2+d^2)^2-
4(ad-bc)^2}}{2}}.$$
\end{proof} \pmn The change of
variables simplifies these formulas. Put
$$a:=x+y,\, d:=x-y,\, b:=z+t,\, c:=z-t.$$ In new variables
$$k(\varphi)=\sqrt{\frac{x^2+y^2+z^2+t^2-\sqrt{(x^2+y^2+z^2+t^2)^2
-(x^2+z^2-y^2-t^2)^2}} {x^2+y^2+z^2+t^2-\sqrt{(x^2+y^2+z^2+t^2)^2
-(x^2+z^2-y^2-t^2)^2}}}.$$ The next change of variables
$$x:=r\sin(\alpha),\,z:=r\cos(\alpha),\,y:=\rho\sin(\beta),\,
t:=\rho\cos(\beta)$$ allows us the further simplification:
$$k(\varphi)=\frac{|r-\rho|}{r+\rho}.$$

\section{Computation of the mean value}

\begin{theor} $\overline{k}_2=3-4\ln(2)$. \end{theor}

\begin{proof}
Without loss of generality we can assume that $|A|>0$, i.e.
$ad-bc>0$. In new variables this condition can be rewritten as
$$x^2+z^2-y^2-t^2>0\text{ or } r>\rho.$$ We will compute the
mean value of $k(\varphi)$ in the domain $ad-bc>0$, i.e. in the
domain
$$r>\rho>0,\, 0\leqslant\alpha\leqslant 2\pi,\,
0\leqslant\beta\leqslant 2\pi.$$ We have
$$
\overline{k}_2=\lim_{R\to\infty}\left(4\pi^2\int_0^R r\,dr\int_0^r
\rho\,\frac{r-\rho}{r+\rho}\, d\rho \left/ 4\pi^2\int_0^R r\,dr
\int_0^r \rho\,d\rho\right.\right)=$$ $$ =\lim_{R\to\infty}\left(
\int_0^R \left.\left(-\frac 12\, \rho^2+2\rho r-2r^2\ln(r+\rho)
\right)\right|_0^r r\,dr \left/ -\frac 18\, r^4\right.\right)=$$
\par\bigskip
$$=3-4\ln 2\approx 0.227411278.$$ \end{proof}

\section{Upper bound for deformation coefficient}
\pn Let $y$ and $z$, $y\geqslant z$, be lengths of vector-columns
of matrix $A$ and $S\leqslant yz$ be the area of parallelogram,
generated by these vectors. Characteristic polynomial of the
matrix $A^t A$ can be written in the following way:
$$s(x)=x^2-(y^2+z^2)\,x+S^2.$$ As
$$p^2,q^2=\frac{y^2+z^2\pm\sqrt{(y^2+z^2)^2-4S^2}}{2}\,,$$ then
$$q^2\leqslant z^2\leqslant y^2\leqslant p^2, \text{ and }\,
k(\varphi)=\frac qp\leqslant \frac zy\,.$$

\begin{theor} $\overline{k}_2<\frac 12$. \end{theor}

\begin{proof} We have
$$\overline{k}_2<\int_0^1 dy\int_0^y \frac zy\, dz \left/
\int_0^1 dy\int_0^y dz\right.=\frac 12\,.$$
\end{proof}

\section{Three-dimensional case}
\pn Let $A$ be the matrix of linear operator $\varphi:\mathbb{R}^3
\to\mathbb{R}^3$, $u,v,w$, $u\leqslant v\leqslant w$, be lengths
vector-columns of this matrix, $S_1,S_2,S_3$ be areas of
parallelograms, generated by pairs of vector-columns and $V$ be
the volume of parallelepiped, generated by vector-columns. Then
characteristic polynomial $s$ of the operator $\varphi^*\varphi$
is of form
$$s(x)=x^3-(u^2+v^2+w^2)\,x^2+(S_1^2+S_2^2+S_3^2)\,x-V^2.$$

\begin{prop} Let $x_1,x_2,x_3$, $x_1\leqslant x_2\leqslant x_3$,
be (real) roots of $s$. Then $x_1\leqslant u^2\leqslant w^2
\leqslant x_3$. \end{prop}

\begin{proof} Computer assisted check. \end{proof} \pn Thus,
$$k(\varphi)=\frac{x_1}{x_3}\leqslant \frac uw\,,$$ and we have the
following estimation of the mean value $\overline{k}_3$ of
three-dimensional deformation coefficient.

\begin{theor} $\overline{k}_3<\frac 13$.\end{theor}

\begin{proof}
$$\overline{k}_3<\int_0^1 dw\int_0^w dv \int_0^v \frac uw\, du \left/
\int_0^1 dw\int_0^w dv \int_0^v du\right.=\frac 13\,.$$
\end{proof}

\begin{rem} Actual value of $\overline{k}_3$ is quite difficult to
compute. We have a rather coarse estimation:
$\overline{k}_3\approx 0.15$. \end{rem}

\vspace{1cm}

\end{document}